\documentclass{elsarticle}
\usepackage{amsthm}
\usepackage{amsmath}
\usepackage{amssymb}
\usepackage{latexsym}
\usepackage{bm}
\usepackage{pstricks}
\usepackage{pst-plot}
\usepackage{psfrag}
\usepackage{graphicx}
\usepackage{ setspace}
\usepackage[all]{xy}

\theoremstyle{plain}
\newtheorem{thm}{Theorem}[section]

\newtheorem{lem}[thm]{Lemma}
\newtheorem{cor}[thm]{Corollary}
\newtheorem{conj}{Conjecture}
\newtheorem{claim}[thm]{Claim}
\theoremstyle{definition}
\newtheorem{defn}{Definition}
\theoremstyle{remark}

\newcommand\p{\mathbb{P}}
\newcommand\OO{\mathcal{O}}
\newcommand\pic{\text{Pic }}
\newcommand\Pic{\text{Pic}}

\newcommand\sym{\text{Sym}}

\title{Rationality of motivic zeta functions for curves with finite abelian group actions}

\author{Justin Mazur}

\address{Indiana University, 831 E. Third St., Bloomington, IN, 47405}

\begin{document}

\begin{abstract}

Let $\mathfrak{Var}_k^G$ denote the category of pairs
$(X,\sigma)$, where $X$ is a variety over $k$ and $\sigma$ is a
group action on $X$.  We define the Grothendieck ring for varieties
with group actions as the free abelian group of isomorphism classes
in the category $\mathfrak{Var}_k^G$ modulo a cutting and pasting
relation.  The multiplication in this ring is defined by the fiber
product of varieties.  This allows for motivic zeta-functions for
varieties with group actions to be defined.  This is a formal power
series $\sum_{n=0}^{\infty}[\text{Sym}^n (X,\sigma)]t^n$ with
coefficients in the Grothendieck ring.  The main result of this
paper asserts that the motivic zeta-function for an algebraic
curve with a finite abelian group action is rational.  This is a
partial generalization of Weil's First Conjecture.  AMS Classifications: 11, 14, 19.
\end{abstract}


\maketitle

\section{Introduction}

Let $X$ be a variety over a finite field $\mathbb{F}$ and let Sym$^n X$ be the $n^{\text{th}}$ symmetric power of $X$.  B. Dwork's proof of Weil's First Conjecture states that the zeta-function
\[ Z_X (t) = \sum_{i=0}^{\infty} |\text{Sym}^n(X)(\mathbb{F})|t^n \]
is a rational function in t \cite{Dwo}.  Kapranov attempted to generalize this result to a general field $k$, by using the Grothendieck ring for varieties over $k$ \cite{Ka}.  Explicitly, let $K_0[\mathfrak{Var}_{k}]$ denote the ring of $\mathbb{Z}$-combinations of isomorphism classes of $k$-varieties modulo the cutting-and-pasting relation
\[ [X] = [Y] + [X \setminus Y] \] for closed $k$-varieties $Y \subset X$.  Kapranov asked whether the motivic zeta-function
\[ \zeta_X(t) = \sum_{i=0}^{\infty} [\text{Sym}^n(X)]t^n \] is rational as a power series in the ring $K_0[\mathfrak{Var}_{k}]$.  While Kapranov was able to prove this when X is a curve, this was shown not to hold for higher dimensional varieties by M. Larsen and V. Lunts \cite{Lu}.  N. Takahashi then conjectured a generalization of this result for curves.  He proposed that you could define a motivic zeta function for curves with finite cyclic group actions and that this motivic zeta function should be rational \cite{Tak}.

In this paper, we will establish an improvement of this conjecture by proving the rationality of motivic zeta functions for curves with finite abelian group actions.  Given an algebraic group $G$, we will construct the Grothendieck ring for varieties with $G$-actions in a way analogous to the construction of the Grothendieck ring for varieties.  From this, we may define motivic zeta-functions and investigate their rationality.  The main result of this paper is the following:

\begin{thm}  Let $G$ be a finite abelian group of order $r$, let $C$ be a non-singular projective curve over an algebraically closed field $k$ of characteristic $0$ or of positive characteristic $p$ with $p \nmid r$, and let $\sigma : G\times C \rightarrow C$ be a group action on $C$.  Then the motivic zeta function
\[ \zeta_{(C,\sigma)}(t) = \sum_{n=0}^{\infty} [\text{Sym}^n (C,\sigma)]t^n\] is rational.
\end{thm}

To prove this we study the Picard bundle, Sym$^n X$ over
Pic$^n X$ with fiber $\mathbb{P}^{n-g}$.  This is the
projectivization of a vector bundle $E_n$ over Pic$^n X$.  As $n$
gets large, we can factor out affine spaces inside of $E_n$ which
have a group action given by the regular representation of $G$.
This allows us to break the vector bundles $E_n$ into products in
the Grothendieck ring for varieties with group actions.  Ultimately,
we are able to use this to break Sym$^n X$ into manageable pieces in
the Grothendieck ring as well.  This leads to representing the
motivic zeta function as a rational function.
\bigskip

\section{The Grothendieck ring for varieties with $G$-actions}

Let $G$ be a fixed algebraic group and let $\mathfrak{Var}^G_k$ denote
the category of varieties with $G$-actions.  The objects of this
category are pairs $(X,\sigma)$, where $X$ is an object of the category of $k$-varieties 
,$\mathfrak{Var}_k$, and $\sigma:G\times X\rightarrow X$ is an algebraic group
action of $G$ on $X$.  Morphisms of this category are
$G$-equivariant variety morphisms.  
 
\begin{defn}  The \emph{Grothendieck group for varieties with G-actions} is a free
abelian group modulo a single cutting and pasting relation given by $K_0(\mathfrak{Var}^G_k)=$ 
\[ \frac{\bigoplus \mathbb{Z}\cdot(\text{isomorphism class in }\mathfrak{Var}^G_k)}{
\langle [X,\sigma] - [X\backslash Y,\sigma] - [Y,\tau] |
(Y,\tau)\text{ closed $G$-invariant subspace of }(X,\sigma) \rangle}.\]  We define
multiplication by
\[ [X,\sigma][Y,\tau]:=[X\times Y,\sigma \times \tau] \]
with
\[
\begin{array}{cccc}
    \sigma\times \tau : & G\times X\times Y & \rightarrow & X\times Y \\
     & (g,x,y) & \mapsto & (\sigma(g,x),\tau(g,y)) \\
\end{array}
\]
making $K_0(\mathfrak{Var}^G_k)$ into a ring called the
\emph{Grothendieck ring for varieties with G-actions}.
\end{defn}

Now we can give the definition of a motivic zeta-function in this context.

 \begin{defn}  Let $(X,\sigma)\in \mathfrak{Var}^G_k$ and define
\[
\begin{array}{cccc}
  \bar{\sigma}: & G\times \sym^n(X) & \rightarrow & \sym^n(X) \\
   & (g,\sum P_i) & \mapsto & \sum \sigma(g,P_i) \\
\end{array}.
\]
Then define
\[ \sym^n(X,\sigma):=(\sym^n(X),\bar{\sigma}) \]
\end{defn}

\begin{defn}  Let $(X,\sigma)$ be an object of
$K_0(\mathfrak{Var}^G_k)$.  We define the \emph{motivic zeta
function of $(X,\sigma)$} by
\[ \zeta_{(X,\sigma)}(t):= \sum_{n=0}^{\infty}[\sym^n(X,\sigma)]t^n, \]
a formal power series in $K_0(\mathfrak{Var}^G_k)[[t]]$.
\end{defn}

\bigskip

\section{Rationality of motivic zeta-functions}

To investigate the rationality of motivic zeta-functions we must have an appropriate definition of rationality.

\begin{defn} Given a commutative ring $A$, a power series $f(t)\in A[[t]]$ is \emph{rational} if there exist polynomials $g(t), h(t) \in A[[t]]$ such that $f(t)$ is the unique solution of $g(t) x = h(t)$.  That is $x$ can be written formally as $x = \frac{h(t)}{g(t)}$.
\end{defn}

Before proving the main result of the paper, we will investigate the rationality of motivic zeta-functions for affine spaces with finite abelian group actions.  We start with the affine line.

\begin{lem}
Let $G$ be a finite group and let $\lambda : G\times \mathbb{A}^1
\rightarrow \mathbb{A}^1$ be a $1$ dimensional linear representation
of $G$ over an algebraically closed field $k$.  Then
\[ \zeta_{(\mathbb{A}^1,\lambda)}(t)=
\sum_{n=0}^{\infty}[\sym^n(\mathbb{A}^1,\lambda)]t^n\] is rational.
\end{lem}

\begin{proof}
First notice that the following map
\[ \begin{array}{cccc}
     \phi : & \sym^n \mathbb{A}^1 & \rightarrow & \mathbb{A}^n \\
      & P_1 + \ldots + P_n & \mapsto & ( P_1 + \ldots + P_n, \sum_{i < j} P_i P_j, \ldots , P_1 P_2 \ldots P_n )
   \end{array}\]
is an isomorphism using the elementary symmetric functions in $n$
variables. This means that in the Grothendieck ring for varieties
with $G$-actions
\[ [\sym^n( \mathbb{A}^1,\lambda)]=[\mathbb{A}^1,\lambda][\mathbb{A}^1,\lambda^2]\ldots [\mathbb{A}^1,\lambda^n].\]
If the order of $G$ is $r$, then we also know that $\lambda^r=1$.
Now we can see that $\zeta_{(\mathbb{A}^1,\lambda)}(t)$ is rational:

\begin{eqnarray*}
\zeta_{(\mathbb{A}^1,\lambda)}(t) & = & \sum_{n=0}^{\infty}[\sym^n(\mathbb{A}^1,\lambda)]t^n \\
& = & \sum_{n=0}^{\infty} \left([\mathbb{A}^1,\lambda][\mathbb{A}^1,\lambda^2]\ldots [\mathbb{A}^1,\lambda^n]\right)t^n\\
& = & \sum_{k=0}^{r-1} \sum_{n=0}^{\infty} [\mathbb{A}^1,\lambda]\ldots[\mathbb{A}^1,\lambda^k] [\sym^r (\mathbb{A}^1, \lambda)]^n t^{k+n} \\
& = & \sum_{k=0}^{r-1} [\mathbb{A}^1,\lambda]\ldots[\mathbb{A}^1,\lambda^k]t^k \left( \sum_{n=0}^{\infty}[\sym^r (\mathbb{A}^1, \lambda)]^n t^n \right) \\
& = & \frac{1}{1-[\sym^r
(\mathbb{A}^1,\lambda)]t}\left(\sum_{k=0}^{r-1}
[\mathbb{A}^1,\lambda]\ldots[\mathbb{A}^1,\lambda^k]t^k \right) .
\end{eqnarray*}
\end{proof}

This allows us to prove the rationality of motivic zeta-functions for general affine spaces with finite abelian group actions.

\begin{thm}  Let $G$ be a finite abelian group of order $r$ and let $\sigma : G \times \mathbb{A}^k \rightarrow \mathbb{A}^k$ be a linear representation.  Then \[ \zeta_{(\mathbb{A}^k,\sigma)}(t)= \sum_{n=0}^{\infty}[\sym^n(\mathbb{A}^k,\sigma)]t^n \] is rational.
\end{thm}

Since $G$ is abelian and $k$ is algebraically closed we have that
\[[\mathbb{A}^k,\sigma]=[\mathbb{A}^1,\lambda_1]\ldots [\mathbb{A}^1,\lambda_k] = [\mathbb{A}^k,(\lambda_1,\ldots , \lambda_k)], \]
where $\lambda_i$ are $1$ dimensional representations of $G$.

\begin{claim} $[\sym^n (\mathbb{A}^k,(\lambda_1,\ldots , \lambda_k))]= [\sym^n(\mathbb{A}^1,\lambda_1)]\ldots [\sym^n(\mathbb{A}^1,\lambda_k)]$.
\end{claim}
\begin{proof}
Let $p: \mathbb{A}^1 \times \mathbb{A}^{k-1} \rightarrow
\mathbb{A}^1$ be the projection map.  This induces a map $p^*:
\sym^n (\mathbb{A}^1 \times \mathbb{A}^{k-1}) \rightarrow \sym^n
(\mathbb{A}^1)$.  B. Totaro proved in \cite{Go} that this map is a
$\sym^n (\mathbb{A}^{k-1})$ bundle on $\sym^n (\mathbb{A}^1)$ with
trivializations in the Zariski topology, giving us that 
\[ [\sym^n (\mathbb{A}^k)]= [\sym^n (\mathbb{A}^1)][\sym^n (\mathbb{A}^{k-1})]\]
in the Grothendieck ring for varieties.

This bundle is $G$-equivariant because the following diagram
commutes:
\[ \xymatrix{
G \times \sym^n (\mathbb{A}^1 \times \mathbb{A}^{k-1}) \ar[rr]^{(\lambda_1,\ldots,\lambda_k)} \ar[d]^{id_G \times p^*} & & \sym^n (\mathbb{A}^1 \times \mathbb{A}^{k-1}) \ar[d]^{p^*} \\
G \times \sym^n (\mathbb{A}^1) \ar[rr]^{\lambda_1} &  & \sym^n
(\mathbb{A}^1)}. \]

But since the action of $G$ on $\mathbb{A}^k$ reduces into $1$
dimensional subrepresentations, we see that
\[ [\sym^n (\mathbb{A}^k,(\lambda_1,\ldots ,\lambda_k) )]= [\sym^n (\mathbb{A}^1,\lambda_1)][\sym^n (\mathbb{A}^{k-1},(\lambda_2,\ldots ,\lambda_k))].\]
We continue this process inductively to get the desired result.
\end{proof}

Now we prove that $\zeta_{(\mathbb{A}^k,\sigma)}(t)=
\sum_{n=0}^{\infty}[\sym^n(\mathbb{A}^k,\sigma)]t^n$ is rational
just as we did when $k=1$.
\begin{eqnarray*}
\zeta_{(\mathbb{A}^k,\sigma)}(t) & = & \sum_{n=0}^{\infty}[\sym^n(\mathbb{A}^k,\sigma)]t^n\\
& = & \sum_{n=0}^{\infty} [\sym^n(\mathbb{A}^1,\lambda_1)]\ldots [\sym^n(\mathbb{A}^1,\lambda_k)] t^n\\
& = & \frac{1}{1- [\sym^n (\mathbb{A}^k,\sigma)]t} \left(
\sum_{l=0}^{r-1} \prod_{i=1}^k [\mathbb{A}^1,\lambda_i] \ldots
[\mathbb{A}^1,\lambda_i^l]t^l \right) .
\end{eqnarray*}

We now reach the main result of the paper: to prove the rationality of motivic zeta-functions for curves with finite abelian group actions.  The remainder of this paper will dedicated to this proof.

\begin{thm}
\label{thm:main} Let $G$ be a finite abelian group of order $r$, let
$C$ be a curve of genus $g$ over an algebraically closed field $k$
of characteristic $0$ or of positive characteristic $p$ with $p
\nmid r$, and let $\sigma : G\times C \rightarrow C$ be a $G$ action
on $C$. Then $\zeta_{(C,\sigma)}(t) = \sum_{n=0}^{\infty} [\sym^n
(C,\sigma)]t^n$ is rational.
\end{thm}

\section{Construction of a vector bundle associated to $\sym^n C$}

The first step in proving this is to construct a vector  bundle $E_n$ with a fiber preserving group action $\sigma$ so that
\[ [\mathbb{P}(E_n,\sigma)] = [\text{Sym}^n (C,\sigma)] ,\] as elements of the Grothendieck ring for varieties with $G$-actions.

Fix $n>2g-2$ and let $\mathcal{L}$ be the Poincar\'{e} (or universal) line bundle
on $C\times \Pic^n C$, with $\mathcal{L}|_{C\times \{[D]\} } \cong \mathcal{O}(D)$ as a line bundle on $C\times \{[D]\}$.  If $p:C \times
\Pic^n C \rightarrow \Pic^n C$ is the projection map, $E_0 = p_*
\mathcal{L}$ is a vector bundle on $\Pic^n C$ so that the
projectivization of $E_0$ is the Picard bundle, i.e. the $\p^{n-g}$
bundle $\sym^n C$ on $\Pic^n C$ \cite{Ar}.  Note that to consider $\Pic^n C$ as a variety it is identified with $\Pic^0 C$, which is an abelian variety.  Let $\varpi_0 : E_0 \rightarrow
\Pic^n C$ denote this vector bundle.  Given a point $[D]$ in $\Pic^n
C$ we have that the fiber of this bundle over $[D]$ is given by \[
\varpi_0^{-1}([D]) = H^0(C,\OO(D)) .\] 

Similarly for all positive integers $m$, we can construct a vector bundle on $\Pic^{n+rm} C$ whose
projectivization is $\sym^{n+rm} C$.  Without loss of generality, assume there is a point $P\in C$ so that
the orbit of $P$, $\{\sigma(g,P)\}=\{gP\}_{g\in G}$, is a set of $r$
distinct points of $C$.  Take note that the curve $C$ does in fact
have points because the field $k$ is algebraically closed.  If there
is no such point, then there is a maximal non-trivial normal
subgroup $H$ of $G$ so that $\sigma$ factors through the quotient
group $G/H$.  If $\sigma' : G/H \times C \rightarrow C$ is this
group action and $\zeta_{(C,\sigma ')}(t)$ is rational, then
$\zeta_{(C,\sigma )}(t)$ must also be rational.

Using this point $P$, we define the following map for every $m$,
\[ \begin{array}{cccc}
     \phi_m : & \Pic^n C & \rightarrow & \Pic^{n+rm} C \\
      & [D] & \mapsto & [D + m \sum_{g\in G} \sigma(g,P)]
   \end{array}.\]
This map is a $G$-equivariant isomorphism between $\Pic^n C$ and $\Pic^{n+rm} C$.  For notational simplicity, from now on we will write $g\cdot P = \sigma(g, P)$ for an element of the group acting on a point $P$ of the curve.  If $E_m$ is a vector bundle on $\Pic^{n+rm} C$ whose projectivization is $\sym^{n+rm} C$, we define $\varpi_m : E_m \rightarrow \Pic^n C$ as a vector bundle by composing with $\phi_m^{-1}$.  
\[ \xymatrix{
E_m \ar[d] \ar[rd]^{\varpi_m} &  \\
\Pic^{n+rm} C \ar[r]^{\phi_m^{-1}} & \Pic^n C }. \]
Given a point $[D]$ in $\Pic^n C$ we have
that the fiber of this bundle over $[D]$ is given by
\[ \varpi_m^{-1}([D]) = H^0 (C, \OO(D + m \sum_{g\in G} g\cdot P)) .\]
Summarizing, we have defined a collection of vector bundles $E_m$
over $\Pic^n C$ so that
\[ \mathbb{P}E_m \cong \sym^{n+rm} C.\]

Now we define a group action on these vector bundles which are
compatible with the group actions on $\sym^{n+rm} C$.  If $x$ is an
element of the vector bundle $E_m$, then it is an element of the
vector space $H^0(C,\OO(D + m\sum g\cdot P))$ over some point $[D]$ in
$\Pic^n C$.  Given this we define $\sigma_m : G\times E_m
\rightarrow E_m$ in the following way.  Given $h$ in $G$ let
$\sigma_h := \sigma(h, - ):C \rightarrow C$ be an automorphism of
$C$.  Then define $\sigma_m(h,f)=f\circ \sigma_{h^{-1}}$, where $f$
is a global section of $\OO(D+ m \sum g\cdot P)$.  It is clear that this
defines a group action on $E_m$.

\begin{claim}
The group action $\sigma_m: G\times E_m \rightarrow E_m$ preserves
fibers of the vector bundle $\varpi_m: E_m \rightarrow \Pic^n C$.
\end{claim}
\begin{proof}
To prove this we must prove that the following diagram commutes:
\[ \xymatrix{
G \times E_m \ar[d]^{id_G \times \varpi_m} \ar[r]^{\sigma_m} & E_m \ar[d]^{\varpi_m} \\
G \times \Pic^n C \ar[r]^{\tilde{\sigma}} & \Pic^n C }, \]
where $\tilde{\sigma}$ is the group action on $\Pic^n C$ defined by
$\tilde{\sigma}(h, \sum n_i P_i) = \sum n_i (h\cdot P_i)$.

If $f \in H^0(C, \OO(D+ m \sum g\cdot P))$ and $h \in G$, then
$\sigma_m(h,f)=f\circ \sigma_{h^{-1}}$.  We need to show that
$f\circ \sigma_{h^{-1}}$ is a global section of $\OO(h\cdot D + m
\sum g\cdot P)$.  We start by considering the divisor of $f$, say $(f) =
\sum P_i - \sum Q_i$, where each $P_i$ is a root of $f$ and each
$Q_i$ is a pole of $f$.  Now, we see that
 \[(f \circ \sigma_{h^{-1}}) =  \sum h\cdot P_i - \sum h\cdot Q_i,\] 
meaning that the zeros
of $f \circ \sigma_{h^{-1}}$ are $h\cdot P_i$ and the poles are $h\cdot Q_i$.
But since $f \in H^0(C, \OO(D+ m \sum g\cdot P))$, we have that 
\[(f) + D+ m \sum_{g\in G} g\cdot P \geq 0,\] 
so that 
\[\sum P_i - \sum Q_i + D+ m \sum_{g\in G} g\cdot P\geq 0.\]  
But from this we get that
\begin{eqnarray*}
& & (f\circ \sigma_{h^{-1}}) + h\cdot D + m \sum_{g\in G} g\cdot P  \\
& = & \sum h\cdot P_i - \sum h\cdot Q_i + h\cdot D + m \sum_{g\in G} g\cdot P  \\
& = & h\cdot \left(\sum P_i - \sum Q_i + D + m \sum_{g\in G} g\cdot P \right) \geq 0,
\end{eqnarray*}
because if $D'$ is an effective divisor then so is $h \cdot D'$.
Therefore,
\[ \sigma_m(h,f)=f\circ \sigma_{h^{-1}} \in H^0 (C, \OO(h\cdot D + m \sum gP)).\]

From this we conclude that $E_m$ is a $(n-g+1+rm)$ dimensional
vector bundle on $\Pic^n C$ with a fiber preserving $G$ action on
it.
\end{proof}

This means that $E_m$ is a $G$-equivariant vector bundle on $\Pic^n
C$.  Now we would like to check that this group action is compatible
with the $G$ action on $\sym^{n+rm} C$.

\begin{claim}
$ [ \mathbb{P}(E_m,\sigma_m) ] = [\sym^{n+rm} C, \overline{\sigma}]
.$
\end{claim}
\begin{proof}
We must show that there is a $G$-equivariant isomorphism between \newline
$\mathbb{P}(E_m,\sigma_m)$ and $(\sym^{n+rm} C, \overline{\sigma})$.

Let $[D]\in \Pic^n C$, let $h\in G$, and let $f\in H^0(C,\OO(D +
m\sum gP))$.  Consider the map
\[ \begin{array}{cccc}
     q: & H^0(C,\OO(D + m\sum gP))\setminus \{0\} & \rightarrow & |D+ m\sum gP| \\
      & f & \mapsto & (f)_0
   \end{array}, \]
where $(f)_0$ denotes the divisor of zeros of $f$.  This map is onto
and $q(f)=q(af)$ for all $a \in k^*$.  Thus it induces a bijection
between the projectivization of $H^0(C,\OO(D + m\sum gP))$ and $|D+
m\sum gP|$.  If $(f)_0= P_1 + \ldots P_l$, it is clear that
$(\sigma_m(h,f))_0= hP_1 +\ldots +hP_l$, thus $q$ is $G$-equivariant.  Therefore, we have a $G$-equivariant
isomorphism between $\mathbb{P}E_m$ and $\sym^{n+rm} C$ with their
respective $G$ actions.
\end{proof}

Now that we have constructed equivariant vector bundles
corresponding to the symmetric powers of $C$, we would like to find
a way to relate these vector bundles to each other with the aim of proving rationality.

\section{Relating the vector bundles $E_i$}

\begin{lem}Let $\tau$ be the regular representation of $G$, then as elements of $K_0(\mathfrak{Var}_k^G)$,
\[ [E_1,\sigma_1]=[E_0,\sigma_0][\mathbb{A}^r,\tau].\]
\label{prop:regular}
\end{lem}

To prove this we will construct an equivariant
split short exact sequence of vector bundles of the form
\[ 0\rightarrow E_0 \rightarrow E_1 \rightarrow F \rightarrow 0 .\]
This will be used to write the $n+r$ dimensional vector bundle $E_1$
as a direct sum of the $n$ dimensional vector bundle $E_0$ and an
$r$ dimensional vector bundle $F$.  When we pass to a trivialization
of these vector bundles, $F$ will have a $G$ action given by the
regular representation on $\mathbb{A}^r$.

\subsection{An equivariant short exact sequence}

First we construct a short exact sequence of sheaves on $C\times
\Pic^n C$.  Since $\{P\}\times \Pic^n C$ is a divisor on $C\times
\Pic^n C$ we get that
\[ 0 \rightarrow \OO_{C\times \Pic^n C}(-\{P\}\times \Pic^n C) \rightarrow \OO_{C\times \Pic^n C} \rightarrow \OO_{\{P\}\times \Pic^n C} \rightarrow 0 \]
is a short exact sequence.  Thus we have the short exact sequence
\[  0 \rightarrow \OO_{C\times \Pic^n C}(-\sum_{g\in G}\{gP\}\times \Pic^n C) \rightarrow \OO_{C\times \Pic^n C} \rightarrow \bigoplus_{g\in G}\OO_{\{gP\}\times \Pic^n C} \rightarrow 0. \]
Tensoring by $\OO_{C\times \Pic^n C}(\sum_{g\in G}\{gP\}\times
\Pic^n C)$ we get
\[ 0 \rightarrow \OO_{C\times \Pic^n C} \rightarrow \OO_{C\times \Pic^n C}(\sum_{g\in G}\{gP\}\times \Pic^n C) \rightarrow \bigoplus_{g\in G}\OO_{\{gP\}\times \Pic^n C} \rightarrow 0.\]
If $\mathcal{L}$ is the Poincar\'{e} bundle on $C\times \Pic^n C$,
and we tensor this short exact sequence with $\mathcal{L}$ we get
the following short exact sequence
\[ 0 \rightarrow \mathcal{L} \rightarrow \mathcal{L} \otimes \OO(\sum_{g\in G}\{gP\}\times \Pic^n C) \rightarrow \bigoplus_{g\in G} \mathcal{L}_{\{gP\}\times \Pic^n C} \rightarrow 0. \]
Finally if $p:C\times \Pic^n C \rightarrow \Pic^n C$ is the
projection map, we apply $p_*$ to the short exact sequence to get
that
\[ 0 \rightarrow p_*\mathcal{L} \rightarrow p_*\left(\mathcal{L} \otimes \OO(\sum_{g\in G}\{gP\}\times \Pic^n C)\right) \rightarrow \bigoplus_{g\in G} p_*\mathcal{L}_{\{gP\}\times \Pic^n C} \rightarrow 0 \]
is a short exact sequence.
\bigskip

By definition, $E_0 = p_* \mathcal{L}$.  Additionally, 
\[ E_1 \cong p_*\left(\mathcal{L} \otimes \OO(\sum_{g\in G}\{gP\}\times \Pic^n C)\right).\]
Indeed, the fiber of the bundle
$p_*\left(\mathcal{L} \otimes \OO(\sum_{g\in G}\{gP\}\times \Pic^n
C)\right)$ over a point $[D]$ is the vector space
\[ H^0 ( C , \OO(D) \otimes_{\OO_C} \OO(\sum_{g\in G} gP)) \cong H^0 (C, \OO(D + \sum_{g\in G} gP)), \]
just as with $E_1$.
Therefore we may rewrite the above short exact sequence as
\[ 0\rightarrow E_0 \rightarrow E_1 \rightarrow \bigoplus_{g\in G} p_*\mathcal{L}_{\{gP\}\times \Pic^n C} \rightarrow 0. \]

Now we want to consider this short exact sequence with groups acting
on each of the vector bundles and prove that the maps are
equivariant with respect to these actions.  Recall that on $E_0$ and
$E_1$ we have $G$ actions defined by the maps $\sigma_0$ and
$\sigma_1$ respectively.

Since $\bigoplus_{g\in G} p_*\mathcal{L}_{\{gP\}\times \Pic^n C}$ is
the direct sum of $r$ line bundles on $\Pic^n C$, we have that
$\bigoplus_{g\in G} p_*\mathcal{L}_{\{gP\}\times \Pic^n C}$ is an
$r$ dimensional vector bundle on $\Pic^n C$ with the fiber over each
$[D]$ isomorphic to the vector space
\[  \bigoplus_{g\in G} H^0 (C, \OO(D) \otimes_{\OO_C} k_{g\cdot P} ), \]
where $k_{g\cdot P}$ denotes the skyscraper sheaf $k$ on $C$ at the point
$g\cdot P$.  To simplify notation, we will denote this vector bundle by
$\varrho : F \rightarrow \Pic^n C$.

Now we define a $G$ action on this vector bundle.  Let
$[D]\in \Pic^n C$, $f \in H^0(C, \OO(D) \otimes_{\OO_C} k_{g\cdot P} )$,
and $h \in G$.  The action $\rho : G\times F \rightarrow F$ is
defined by $\rho(h,f)= f\circ \sigma_{h^{-1}}$, where
$\sigma_{h^{-1}}$ is the automorphism of $C$ defined by $h^{-1}$ in
$G$.  It is clear that given a section $f \in H^0(C, \OO(D)
\otimes_{\OO_C} k_{g\cdot P} )$ over $[D]$, $f\circ \sigma_{h^{-1}}$ is a
section in $H^0(C, \OO(h\cdot D) \otimes_{\OO_C} k_{hg\cdot P} )$ over $[h\cdot D]$,
thus this defines a fiber preserving group action on $F$.

\begin{claim}
\[\xymatrix{
0 \ar[r] & E_0 \ar[r]^i & E_1 \ar[r]^q &  F \ar[r] & 0} \] is a
$G$-equivariant short exact sequence of locally free sheaves on
Pic$^n C$.
\end{claim}
\begin{proof}
To show that the short exact sequence is also $G$-equivariant, we
must show that the following diagram commutes:
\[ \xymatrix{
0 \ar[r]  & G\times E_0 \ar[r]^{id_G\times i} \ar[d]^{\sigma_0} & G\times E_1 \ar[r]^{id_G\times q} \ar[d]^{\sigma_1} & G\times F\ar[r] \ar[d]^{\rho} & 0  \\
0 \ar[r] & E_0 \ar[r]^i & E_1 \ar[r]^q &  F \ar[r] & 0}. \]

The left square of this diagram commutes.  Indeed, let $h\in G$ and
$x\in E_0$.  If $\varpi_0(x)=[D]$ in $\Pic^n C$, we have that $x$ is
a section $f \in H^0(C,\OO(D))$.  So we have
\begin{eqnarray*}
i \circ \sigma_0(h,f) & = & i\circ f\circ \sigma_{h^{-1}}\\
& = & f\circ \sigma_{h^{-1}},
\end{eqnarray*}
which means that
\begin{eqnarray*}
\sigma_1\circ (id_G\times i) (h,f) & = & \sigma_1 (h,f)= f\circ \sigma_{h^{-1}}\\
& = & i \circ \sigma_0(h,f),
\end{eqnarray*}
proving that $\sigma_1\circ (id_G\times i) = i \circ \sigma_0$.

The right square of this diagram also commutes.  Indeed, let $h\in
G$ and $x\in E_1$.  If $\varpi_1(x)=[D]$ in $\Pic^n C$, we have that
$x$ is a section $f \in H^0(c,\OO(D+ \sum g\cdot P))$.  First notice that
$q: E_1 \rightarrow F$ is the evaluation map so that $q(f)$ is equal
to the $r$-tuple $(f(g\cdot P))_{g\in G}$.  So we have
\begin{eqnarray*}
q\circ \sigma_1 (h,f) & = & q (f\circ \sigma_{h^{-1}})\\
& = & (f\circ \sigma_{h^{-1}}(g\cdot P))_{g\in G},
\end{eqnarray*}
which means that
\begin{eqnarray*}
\rho \circ (id_G \times q)(h,f) & = & \rho (h, (f(gP))_{g\in G})\\
& = & (f\circ \sigma_{h^{-1}}(gP))_{g\in G}\\
& = & q\circ \sigma_1 (h,f),
\end{eqnarray*}
proving that $\rho \circ (id_G \times q) = q\circ \sigma_1 (h,f)$.
\end{proof}
\bigskip

Next, we find an affine $G$-invariant open subset of $\Pic^n C$ for which the vector bundles $E_0$, $E_1$, and $F$ are trivial over it.
Let $U_0, U_1, U_2$ be an affine open subsets of $\Pic^n C$ so that
$\varpi_0^{-1}(U_0)\cong U_0\times \mathbb{A}^{n+1-g}$,
$\varpi_1^{-1}(U_1)\cong U_1\times \mathbb{A}^{n+1-g + r}$, and
$\varrho^{-1}(U_2)\cong U_2\times \mathbb{A}^r$.  If $W = U_0 \cap
U_1 \cap U_2$, $W$ is an affine non-empty open subset of $\Pic^n C$
and each of the above vector bundles are trivial over it.  Note that
$W$ is non-empty because $\Pic^n C$ is an irreducible variety.  Then
let

\begin{eqnarray}
\label{eq:trivial} U & = & \bigcap_{g\in G} g\cdot W
\end{eqnarray}
which is non-empty and open because $G$ is finite and $\Pic^n C$ is
irreducible.  This $U$ has all of our desired properties, it is an
affine $G$-invariant open subset of $\Pic^n C$ so that each of the
vector bundles are trivial over it.

Restricting the equivariant short exact sequence to the vector
bundles on $U$, we get a sequence of maps
\[ \xymatrix{
 U\times \mathbb{A}^{n-g+1} \ar[r]^{i} & U\times
\mathbb{A}^{n-g+1+r} \ar[r]^{q} & U\times \mathbb{A}^r },\] which
corresponds to the following short exact sequence of free sheaves on
Pic$^n C$
\[ \xymatrix{
0 \ar[r] & E_0|_U \ar[r]^{i} & E_1|_U\ar[r]^{q} & F|_U \ar[r] &
0}.\]

\subsection{A splitting for the short exact sequence}

Our next step is to show that the above short exact sequence has a
$G$-equivariant splitting.  To do this we follow a variation of the
proof of Maschke's Theorem from basic linear representation theory.

Consider the projection map $P: U\times \mathbb{A}^{n-g+1+r}
\rightarrow U\times \mathbb{A}^{n-g+1}$.  This map is not
$G$-equivariant, so we define an equivariant version of $P$.  Define
$P': U\times \mathbb{A}^{n-g+1+r} \rightarrow U\times
\mathbb{A}^{n-g+1}$ with
\[ P'([D],x) = \frac{1}{r} \sum_{g\in G} \sigma_1(g^{-1},-)\circ P \circ \sigma(g,-)([D],x)\]
for all $[D]\in U$ and $x\in \mathbb{A}^{n-g+1+r}$, so that
\[ P' = \frac{1}{r}\sum_{g\in G} g^{-1}\cdot P \cdot g.\]  Addition makes sense here because $P'([D],x)=([D],x')$ for some $x' \in \mathbb{A}^{n-g+1}$, meaning that this addition takes place in the vector space $\mathbb{A}^{n-g+1}$ over $[D]$ in $U$, and this map is clearly $G$-equivariant and onto.  Notice that since the characteristic of the ground field $k$ does not divide the order of $G$, we have that $r$ is invertible in $k$, so that it makes sense to write $\frac{1}{r}$.

Next, define
\[ K = \left\{ ([D],x) \in U \times \mathbb{A}^{n-g+1+r} : P'([D],x)=([D],0) \right\}.\]
This is a $G$-invariant closed subvariety of $U\times
\mathbb{A}^{n-g+1+r}$.  Indeed, this is a closed subvariety because
it is given by a closed algebraic condition.  It is $G$-invariant
because given $([D],x)\in K$,
\[ P'\sigma_1(h,([D],x)) = \sigma_1(h,P'([D],x)) = \sigma_1(h,([D],0)),\]
and $\sigma_1(h,([D],0)) = ([h \cdot D],0)$ because the zero section of
$H^0(C,\OO(D+\sum g\cdot P))$ maps to the zero section of $H^0(C,\OO(h\cdot D+
\sum g\cdot P)$ under precomposition with $\sigma_{h^{-1}}$, the action of
$h$.

Furthermore, $K$ is an $r$ dimensional vector bundle on $U$ via the
projection map $K \rightarrow U$. Indeed, the map $P':   U\times
\mathbb{A}^{n-g+1+r} \rightarrow U\times \mathbb{A}^{n-g+1}$ of
trivial vector bundles is associated to a map $P'^*: \mathcal{F}
\rightarrow \mathcal{G}$ of locally free sheaves $F$ and $G$
associated to the respective vector bundles.  The sheaf $ker P'^*$
is a locally free sheaf, which is associated to some vector bundle.
But, because of the way $K$ was defined, $K$ is this vector bundle
associated to $ker P'^*$.  In particular, we see that $K$ is in fact
a vector bundle.  Additionally, let $[D]\in U$ and consider the
linear transformation $P'_{[D]} : \mathbb{A}^{n-g+1+r}\rightarrow
\mathbb{A}^{n-g+1}$ with $P'_{[D]}(x)=x'$ where
$P'([D],x)=([D],x')$.  Since $P'_{[D]}$ is onto, we have that the
kernel of $P'_{[D]}$ has dimension $r$.  Therefore, the fiber of $K$
over $[D]$ is the $r$ dimensional vector space ker $P'_{[D]}$,
meaning that $K$ is an $r$ dimensional vector bundle over $U$.

As we found a $G$-equivariant open affine $U$ for which all of the
vector bundles are trivial over $U$, we shall redefine $U$ to
include a trivialization for $K$ as well.

Now let us consider the equivariant map \[q: U\times
\mathbb{A}^{n-g+1+r}\cong U\times \left( \mathbb{A}^{n-g+1} \oplus
\mathbb{A}^r\right) \rightarrow U\times \mathbb{A}^r. \]  From the
exactness of this sequence we know that $q$ maps $U\times
\mathbb{A}^{n-g+1}$ to $U\times \{0\}$.  This means that the
restriction of $q$ to $U\times \mathbb{A}^r$ must be an equivariant
isomorphism.  In particular, we have the desired equivariant
splitting for the short exact sequence. Simply define the splitting
to be the inverse of $q$ restricted to $U\times \mathbb{A}^r$.

Summing up we have found a decomposition of $U\times
\mathbb{A}^{n-g+1+r}$ into equivariant summands, i.e.
\[ U\times \mathbb{A}^{n-g+1+r} \cong U\times \left( \mathbb{A}^{n-g+1} \oplus \mathbb{A}^r\right).\]

\subsection{The regular representation of $G$}

Finally, we must analyze the vector bundle $F$.  We will find that
when restricted to the trivialization $U$, $F\cong U\times
\mathbb{A}^r$ is an equivariant isomorphism where the action of $G$
on $\mathbb{A}^r$ is the regular representation.

Let us examine the vector bundle $\varrho : F \rightarrow \Pic^n C$
carefully.  If we restrict our attention to $U=$ Spec $A$ in $\Pic^n
C$ for some $k$-algebra $A$, we see that this vector bundle is associated to some
$G$-equivariant $A$-module $M$ of rank $r$, because the vector
bundle $\varrho : F \rightarrow \Pic^n C$ preserves fibers under the
action of $G$.

\begin{lem}
$M$ is isomorphic to $A\otimes_k k[G]$ as a $G$-equivariant
$A$-module.
\end{lem}
\begin{proof}
We start by investigating the $G$ action on the $A$-module $M$.
Since $M$ is a $G$-equivariant $A$-module, we have the following property:
\[ g \cdot am = (g\cdot a)(g\cdot m) \text{ for all }a\in A, m\in M, g\in G.\]

Since the rank of $M$ is $r$ and the order of $G$ is $r$, we can let
the set $\{x_g\}$ be a basis for $M$ so that
\[M\cong \oplus_{g\in G} Ax_g.\]
We now compute the $G$ action on $M$.

\begin{claim} $g\cdot x_{e} =  x_{g}$.
\end{claim}
\begin{proof}
Recall from the definition of $\rho :G\times F \rightarrow F$ that
given a global section $f \in H^0(C, \OO(D)\otimes k_{P})$,
$\rho(g,f)=f\circ \sigma_{g^{-1}}$ is a global section in
$H^0(C,\OO(gD)\otimes k_{gP})$.  This means that
\[g\cdot  x_{e} = a_g x_{g},\text{ for some }a_g\in A.\]
Also, since
\begin{eqnarray*}
 x_e & = & (g^{-1}g) \cdot x_e\\
& = & g^{-1}\cdot (a_g x_g) \\
& = & (g^{-1}\cdot a_g)(g^{-1}\cdot x_g ) ,
\end{eqnarray*}
we know that $g^{-1}\cdot a_g$ must be invertible in $A$ with
\[  g^{-1}\cdot x_g = (g^{-1}\cdot a_g)^{-1} x_e .\]
Additionally, we know that the action of $g^{-1}$ on $A$ is a ring
automorphism.  Therefore, $a_g$ is invertible in $A$.

Thus, without loss of generality we may assume that $a_g=1$ for all
$g\in G$.  Indeed, the following map:
\[ \begin{array}{cccc}
     \phi : & M=\bigoplus_{g\in G} Ax_g & \rightarrow & M=\bigoplus_{g\in G} Ax_g \\
      & x_g & \mapsto & a_g^{-1}x_g
   \end{array}\]
defines a $G$-equivariant $A$-module isomorphism.

Finally, we conclude that
\[ g\cdot   x_{e} =  x_{g}\]
as desired.
\end{proof}

\bigskip

Let $g,h\in G$, and let $a\in A$.  We can now compute the action of
$g$ on a general element of $M$.
\begin{eqnarray*}
g\cdot (a x_h) & = & (g\cdot a)(g\cdot  x_h )\\
& = & (g\cdot a) ((gh) \cdot x_e) \\
& = & (g\cdot a) x_{gh} \\
& = & (g\cdot a) x_{gh}.
\end{eqnarray*}

Consider the $G$-equivariant $A$-module, $A \otimes_k k[G]$, where
$k[G]$ denotes the group ring $G$ over $k$.  The $G$-action on this
$A$-module is given as follows:
\[ g \cdot ( a \otimes h) = (g\cdot a) \otimes (gh).\]
From this, it is easy to see that the map
\[ \begin{array}{cccc}
     \phi : & A \otimes_k k[G] & \rightarrow & M \\
      & a\otimes g & \mapsto & a  x_g
   \end{array}\]
is a $G$-equivariant isomorphism of $A$-modules.
\end{proof}

Using the isomorphism $\phi$, we have a $G$-equivariant isomorphism
of schemes given by
\[ \phi^* :U \times_k \mathbb{A}^r \rightarrow \varrho^{-1}(U),\]
and as elements of $K_0( \mathfrak{Var}_k^G)$ we have
\[ [\varrho^{-1}(U),\rho] = [U,\tilde{\sigma}][ \mathbb{A}^r,\tau], \]
where $\tau$ denotes the regular representation of $G$.

\bigskip
\bigskip

We now return to our equivariant short exact sequence of free
sheaves on Pic$^n C$
\[ \xymatrix{
0 \ar[r] & E_0|_U \ar[r]^{i} & E_1|_U\ar[r]^{q} & F|_U \ar[r] &
0}.\] Since this short exact sequence is split, we have an
equiviariant isomorphism
\[ \varphi : U\times \mathbb{A}^{n-g+1+r} \rightarrow U\times \left(\mathbb{A}^{n-g+1}\oplus \mathbb{A}^r\right).\]
We can define an isomorphism using $\varphi$ as follows:
\[ \begin{array}{cccc}
     \psi : & U\times \mathbb{A}^{n-g+1+r} & \rightarrow & U\times \mathbb{A}^{n-g+1}\times \mathbb{A}^r \\
      & ([D],(x,y)) & \mapsto & ([D],x,y)
   \end{array}.\]
Since this map is equivariant we may write
\[ [U\times \mathbb{A}^{n-g+1+r},\sigma_1]= [U\times \mathbb{A}^{n-g+1},\sigma_0][\mathbb{A}^r,\tau] \]
as elements of $K_0( \mathfrak{Var}_k^G)$. Finally, recalling that
$U$ is a trivialization for the vector bundles $\varpi_1:E_1
\rightarrow \Pic^n C$ and $\varpi_0:E_0 \rightarrow \Pic^n C$, we
have that
\begin{eqnarray*}
 [ \varpi_1^{-1}(U),\sigma_1] & = & [U\times \mathbb{A}^{n-g+1+r},\sigma_1]\\
 & = & [U\times \mathbb{A}^{n-g+1},\sigma_0][\mathbb{A}^r,\tau]\\
 & = & [\varpi_0^{-1}(U),\sigma_0][\mathbb{A}^r,\tau].
\end{eqnarray*}

\subsection{Application to the Grothendieck ring}

Finally we apply this local argument to the entirety of the vector
bundles $E_0$ and $E_1$.

Consider $X = \Pic^n C \setminus U$ a closed possibly reducible
subvariety of $\Pic^n C$ of dimension strictly less than $U$.  We
have
\begin{eqnarray*}
[E_1,\sigma_1] & = & [\varpi_1^{-1}(U),\sigma_1] + [\varpi_1^{-1}(X),\sigma_1]\\
& = & [U\times \mathbb{A}^{n-g+1+r},\sigma_1] + [\varpi_1^{-1}(X),\sigma_1]\\
& = & [U\times \mathbb{A}^{n-g+1},\sigma_0][\mathbb{A}^r,\tau] +
[\varpi_1^{-1}(X),\sigma_1].
\end{eqnarray*}

If we can find a dense affine open $G$-invariant $W$ in $X$ so that
the vector bundles $E_0$, $E_1$, and $F$ are trivial on $W$, we can
repeat the argument above to break this sum in the Grothendieck ring down further.

First, we construct the open set $W$.  Let $X_1,\ldots, X_m$ denote
the distinct irreducible components of $X$.  For each $X_i$ we may
consider the vector bundles $E_0$, $E_1$, and $F$ on $X_i$.  Let
$U_i$ be open affine subsets of $X_i$ so that each of the vector
bundles are trivial on $U_i$.  Given $g\in G$, $g\cdot U_j \cap U_i$
is either empty or dense in $X_i$, so given $g\in G$ we define
\[ W_i^g = \bigcap_{\alpha \in N_i^g} g\cdot U_\alpha \subset X_i,\]
where $N_i^g = \{ \alpha \ | \ g\cdot U_{\alpha} \cap U_i \neq
\emptyset \} \subset \{1,2,\ldots,m\}$.  Then define
\[ W_i = \bigcap_{g\in G} W_i^g \subset X_i.\]
Finally, we must make the $W_i$ disjoint.  Notice that the
intersections of the components $X_i$ are closed proper subvarieties
of the $X_i$'s.  To make these disjoint we simply subtract off the
intersections of the $X_i's$ and the $G$ orbits of these closed
subvarieties.  Since $G$ is finite and the $X_i$ are distinct
irreducible components, this is a finite collection of closed proper
subvarieties.  We define sets $W_i'$ to be the sets $W_i$ with these closed proper subvarieties removed.  The $W_i'$ are still  non-empty open affine subsets of $X_i$ with all
of the desired properties.  Now define $W = \cup W_i'$, which is
affine, open, $G$-invariant, and a trivialization for the vector
bundles $E_0$, $E_1$, and $F$.

We repeat the arguments of the previous subsection verbatim to conclude
that
\begin{eqnarray*}
[\varpi_1^{-1}(W),\sigma_1] & = & [W\times \mathbb{A}^{n-g+r+1},\sigma_1]\\
& = & [W\times \mathbb{A}^{n-g+1},\sigma_0] [\mathbb{A}^r,\tau].
\end{eqnarray*}
Thus we have $[E_1,\sigma_1] $
\begin{eqnarray*}
 = & [U\times \mathbb{A}^{n-g+1},\sigma_0][\mathbb{A}^r,\tau] + [W\times \mathbb{A}^{n-g+1},\sigma_0] [\mathbb{A}^r,\tau] + [\varpi_1^{-1}(X\setminus W),\sigma_1] \\
 = & [\mathbb{A}^r,\tau]\left( [\varpi_0^{-1}(U),\sigma_0]+[\varpi_0^{-1}(W),\sigma_0]\right) + [\varpi_1^{-1}(X\setminus W),\sigma_1]\\
 = & [\mathbb{A}^r,\tau] [\varpi_0^{-1}(U \cup W),\sigma_0] +
[\varpi_1^{-1}(X\setminus W),\sigma_1],
\end{eqnarray*}
with $X\setminus W$ a closed subvariety of $X$ of dimension strictly
less than $X$.

We can repeat this process until the closed subvariety we are left
with is a finite collection of points, i.e. a 0 dimensional
subvariety.  Let $Y = \{ Q_1, \ldots, Q_m \}$ be this collection of
points in $\Pic^n C$. This leave us with the following formula in
$K_0(\mathfrak{Var}_k^G)$:
\[ [E_1,\sigma_1] = [\mathbb{A}^r,\tau] [\varpi_0^{-1}(\Pic^n C \setminus Y),\sigma_0] + [\varpi_1^{-1}(Y),\sigma_1] \]
But again, we see that
\begin{eqnarray*}
[\varpi_1^{-1}(Y),\sigma_1] & = & [Y\times \mathbb{A}^{n-g+1},\sigma_0][\mathbb{A}^r,\tau]\\
& = & [\varpi_0^{-1}(Y),\sigma_0][\mathbb{A}^r,\tau].
\end{eqnarray*}
Indeed, since $Y$ is a finite collection of points we have that $Y$
is an affine \newline $G$-invariant subvariety of $\Pic^n C$ with
$\varpi_1^{-1}(Y)\cong Y\times \mathbb{A}^{n-g+1+r}$.  This means
that we may still apply the previous subsections arguments to this
case.

Therefore,
\begin{eqnarray*}
 [E_1,\sigma_1] & = & [\mathbb{A}^r,\tau] \left( [\varpi_0^{-1}(\Pic^n C \setminus Y),\sigma_0] + [\varpi_0^{-1}(Y),\sigma_0] \right) \\
 & = & [\varpi_0^{-1}\left((\Pic^n C \setminus Y)\cup Y\right),\sigma_0][\mathbb{A}^r,\tau]\\
 & = & [\varpi_0^{-1}(\Pic^n C),\sigma_0][\mathbb{A}^r,\tau]\\
 & = & [E_0,\sigma_0][\mathbb{A}^r,\tau],
 \end{eqnarray*}
completing the proof of Lemma~\ref{prop:regular}.

\bigskip

We can immediately conclude the following corollary, whose proof is
obvious.

\begin{cor}
\label{cor:regular} Let $n > 2g-2$, let $\varpi_m : E_m \rightarrow
\Pic^n C$ be the vector bundles described previously for $m \geq 0$,
and let $\tau:G\times \mathbb{A}^r$ denote the regular
representation of $G$. Then $[E_m,\sigma_m] =
[E_0,\sigma_0][\mathbb{A}^r,\tau]^m$ as elements of
$K_0(\mathfrak{Var}_k^G)$.
\end{cor}

This allows to break up the vector bundles $E_m$ into $E_0$ and
copies of regular representations of $G$ in the Grothendieck ring.
This will give us the necessary tools to prove the rationality of
the motivic zeta function once we find a way to apply this
result to the projectivizations of $E_m$.

\section{Relating the vector bundles $E_i$ to the Picard bundle}

Now we need to find a way to apply the above methods to the
projectivizations of the vector bundle $E_m$, the symmetric power, $\sym^{n+rm} C$.  To
do this we will need to assume that the group $G$ is abelian.
Notice that we have not needed this assumption until now.

Let $U$ be the open affine subset of $\pic^n C$ as defined above in
equation~\ref{eq:trivial}.  Recall that the vector bundle $E_m$ is
trivial over $U$, so that we have an equivariant isomorphism
\[ \phi: \varpi_m^{-1}(U) \rightarrow (U\times \mathbb{A}^{n-g+1}) \times (\mathbb{A}^r)^m, \]
giving an equation in $K_0(\mathfrak{Var}_k^G)$ like that in
Corollary~\ref{cor:regular},
\[ [\varpi_m^{-1}(U),\sigma_m ] = [U\times \mathbb{A}^{n-g+1},\sigma_0] [\mathbb{A}^r,\tau]^m.\]

Since the group $G$ is abelian, the regular representation $\tau:
G\times \mathbb{A}^r \rightarrow \mathbb{A}^r$ can be reduced
completely into 1 dimensional representations.  That is to say there
exists 1 dimensional representations $\lambda_1, \lambda_2,\ldots,
\lambda_r$ and an equivariant isomorphism
\[ \psi: \mathbb{A}^r \rightarrow (\mathbb{A}^1)^r ,\]
giving the following equation in $K_0(\mathfrak{Var}_k^G)$:
\[ [\mathbb{A}^r,\tau] = [\mathbb{A}^1,\lambda_1][\mathbb{A}^1,\lambda_2]\ldots [\mathbb{A}^1,\lambda_r].\]

Defining $\Phi = (id_{\varpi_0^{-1}(U)} \times \psi) \circ \phi$, we
get an equivariant isomorphism
\[ \Phi : \varpi_m^{-1}(U) \rightarrow (U\times \mathbb{A}^{n-g+1}) \times (\mathbb{A}^1)^{rm},\]
giving the following equation in $K_0(\mathfrak{Var}_k^G)$:
\[ [\varpi_m^{-1}(U),\sigma_m ] = [U\times \mathbb{A}^{n-g+1},\sigma_0][\mathbb{A}^1,\lambda_1]^m[\mathbb{A}^1,\lambda_2]^m\ldots [\mathbb{A}^1,\lambda_r]^m.\]

Since the fiber bundle $\pi'_m:\sym^{n+rm}C\rightarrow \Pic^n C$ is
the projectivization of the vector bundle $E_m$, $\Phi$ gives rise
to an equivariant isomorphism
\[ \overline{\Phi} : \pi'^{-1}(U) \rightarrow U\times \mathbb{P}(\mathbb{A}^{n-g+1}\times \mathbb{A}^{rm})\cong U\times \mathbb{P}^{n+rm-g}.\]

A point in $\pi'^{-1}(U)$ can thus be represented as
\[ (u, [y_0:y_1:\ldots:y_{n-g}:x_{11}:x_{12}:\ldots :x_{1r}:x_{21}:\ldots :x_{rr}]) \]
with $u\in U$, $y_i,x_{ij}\in k$.  Also assume that for all $g\in
G$,
\[ \overline{\sigma}(g,(u,[y_i:x_{ij}])) = (u',[y_i':\lambda_i(g, x_{ij})]),\]
for some $u'\in U$, $y_i'\in k$, and where $\lambda_i$ are the 1
dimensional representations introduced above.  We may make this
assumption because the map $\Phi$ diagonalized the representation of
$G$.

Let $X_{11}$ be a closed subset of $X:=U\times
\mathbb{P}(\mathbb{A}^{n-g+1}\times \mathbb{A}^{rm})$ defined by
$\{x_{11}=0\}$.  $X_{11}$ and $X\setminus X_{11}$ are both
$G$-invariant because the action of $G$ acts diagonally on the
$x_{ij}$.  In the Grothendieck ring we get
\[ [\pi'^{-1}(U),\overline{\sigma}] = [X,\overline{\sigma}] = [X_{11},\overline{\sigma}]+ [X\setminus X_{11},\overline{\sigma}].\]

We would now like to get a clear picture of the two spaces
$X_{11}$ and $X\setminus X_{11}$.  The space $X_{11}$ is isomorphic to $U\times \mathbb{P}^{n+rm-g-1}$
and its coordinates are the same as $X$, but with the $x_{11}$
coordinate deleted.

To understand what $X\setminus X_{11}$ looks like we define an
isomorphism
\[ \begin{array}{cccc}
     \varphi : & X\setminus X_{11} & \rightarrow & U\times \mathbb{A}^{n-g+1}\times \mathbb{A}^{rm-1} \\
      & (u,[y_i:x_{ij}]) & \mapsto & (u,\frac{y_i}{x_{11}},\frac{x_{ij}}{x_{11}})
   \end{array}. \]
The codomain is the familiar $\varpi_0^{-1}(U)\times
\mathbb{A}^{rm-1}\cong U\times \mathbb{A}^{n-g+1} \times
\mathbb{A}^{rm-1}$.  In order to insure that $\varphi$ is an
equivariant isomorphism, we endow the codomain with an appropriate
$G$ action.  Therefore, in the Grothendieck ring,
\[ [X\setminus X_{11},\overline{\sigma}]  \]
\[ =[U\times \mathbb{A}^{n-g+1},\frac{1}{\lambda_1}\otimes \sigma_0][\mathbb{A}^1,\frac{1}{\lambda_1}\otimes \lambda_1]^{m-1}[\mathbb{A}^1,\frac{1}{\lambda_1}\otimes \lambda_2]^m\ldots [\mathbb{A}^1,\frac{1}{\lambda_1}\otimes \lambda_r]^m .\]  The action $1/\lambda_1 \otimes \lambda_1$ is really the trivial action, and from now on we will write $1$ to denote the trivial group action.

The next step is to break up $X_{11}$ into pieces just as we did
$X$.  We let $X_{12}$ be the closed subset of $X_{11}$ defined by
$\{x_{12}=0\}$.  In the Grothendieck ring, we have that
\[ [X_{12},\overline{\sigma}] = [U\times \mathbb{P}^{n+rm-g-2},\overline{\sigma}] \]
and
\[ [X_{11}\setminus X_{12},\overline{\sigma}]  \]
\[ =[U\times \mathbb{A}^{n-g+1},\frac{1}{\lambda_1}\otimes \sigma_0][\mathbb{A}^1,1]^{m-2}[\mathbb{A}^1,\frac{1}{\lambda_1}\otimes \lambda_2]^m\ldots [\mathbb{A}^1,\frac{1}{\lambda_1}\otimes \lambda_r]^m .\]

We can repeat this process $rm$ times, until we reach the following
equation:
\[ [X,\overline{\sigma}] = [U\times \mathbb{P}^{n-g},\overline{\sigma}]  \]
\[ +\sum_{j=1}^{r} [U\times \mathbb{A}^{n-g+1},\frac{1}{\lambda_j}\otimes \sigma_0] \left(\sum_{i=1}^m [\mathbb{A}^1,1]^{m-i}\prod_{k = j+1}^r[\mathbb{A}^1,\frac{1}{\lambda_j}\otimes \lambda_{k}]^m\right) . \]

This brings us to an important claim which allows us to break up the
space $\sym^{n+rm}C$ in the Grothendieck ring.

\begin{claim}
\label{claim:symbreak}
\[ [\sym^{n+rm} (C,\sigma)] = [\sym^n (C,\sigma)]   \]
\[ +\sum_{j=1}^{r} [E_0,\frac{1}{\lambda_j}\otimes \sigma_0] \left(\sum_{i=1}^m [\mathbb{A}^1,1]^{m-i}\prod_{k = j+1}^r[\mathbb{A}^1,\frac{1}{\lambda_j}\otimes \lambda_{k}]^m\right) .\]
Define \[ \Omega_{ij} :=
[\mathbb{A}^1,1]^{m-i}\prod_{k = j+1}^r[\mathbb{A}^1,\frac{1}{\lambda_j}\otimes \lambda_{k}]^m, \] which simplifies this equation to say
\[ [\sym^{n+rm} (C,\sigma)] = [\sym^n (C,\sigma)] +  \sum_{j=1}^{r} [E_0,\frac{1}{\lambda_j}\otimes \sigma_0] \sum_{i=1}^m \Omega_{ij}  .\]
\end{claim}
\begin{proof}
This proof will follow the same strategy as appears in Section 5.4.  Let $U_0=U$ in $\Pic^n C$ be the open affine trivialization
defined in this section.  We have that
\[ [\sym^{n+rm} C,\overline{\sigma}] = [ \pi_m'^{-1}(U_0),\overline{\sigma}] + [\pi_m'^{-1}(\Pic^n C \setminus U_0),\overline{\sigma}].\]
Since $U_0$ is an open dense subset of $\Pic^n C$, $\Pic^n C
\setminus U$ is a closed subvariety of $\Pic^n C$ of dimension
strictly less than $U_0$.  This variety may not be irreducible, so
we construct a $G$-invariant dense open affine subset of $\Pic^n C
\setminus U$ exactly as we did in Section 5.4, call this open set
$U_1$.  We may repeat previous arguments to
conclude that
\begin{eqnarray*}
 [\pi_m'^{-1}(U_1),\overline{\sigma}] & = & [U_1\times \mathbb{P}^{n+rm-g},\overline{\sigma}]\\
 & = & [U_1\times \mathbb{P}^{n-g},\overline{\sigma}] + \sum_{j=1}^{r} [U_1\times \mathbb{A}^{n-g+1},\frac{1}{\lambda_j}\otimes \sigma_0] \sum_{i=1}^m \Omega_{ij} .
\end{eqnarray*}
Next we define a $G$-invariant open affine subset $U_2$ in the same
way we defined $U_1$, whose compliment has even smaller dimension in
$\Pic^n C$. This process repeats $N$ times until we are left with a
0 dimensional variety, where this equation will still hold.

Since
\[ \sum_{k=0}^N [U_k] = [\Pic^n C], \]
we compute that $[\sym^{n+rm} (C,\sigma)]$
\begin{eqnarray*}
= & \sum_{k=0}^N [ \pi_m'^{-1}(U_k),\overline{\sigma} ] \\
 = & \sum_{k=0}^N \left( [U_k\times \mathbb{P}^{n-g},\overline{\sigma}] + \sum_{j=1}^{r} [U_l\times \mathbb{A}^{n-g+1},\frac{1}{\lambda_j}\otimes \sigma_0] \sum_{i=1}^m \Omega_{ij} \right)\\
  = & [\sym^n (C,\sigma)] +  \sum_{j=1}^{r} [E_0,\frac{1}{\lambda_j}\otimes \sigma_0] \sum_{i=1}^m \Omega_{ij},
\end{eqnarray*}
completing the proof.
\end{proof}

\section{Rationality of the motivic zeta-function}

Finally, we are left with proving the rationality of the power
series.  Let $n_i = 2g + i$ for $i = 1,2,\ldots , r$.  Since $n_i > 2g-2$ for
all $i$, we have that $[\sym^{n_i} C]=[\Pic^{n_i} C][\p^{n_i - g}]$
in the Grothendieck ring for varieties.

Now we rewrite the motivic zeta-function for $(C, \sigma)$ as
\begin{eqnarray*}
 \zeta_{(C,\sigma)}(t) & = & \sum_{n=0}^{\infty} [\sym^n (C,\sigma)]t^n  \\
 & = & \sum_{n=0}^{2g}[\sym^n (C,\sigma)]t^n + \sum_{i=1}^r t^{n_i} \left(\sum_{m=0}^{\infty} [\sym^{n_i+rm} (C,\sigma)] t^{rm} \right) .
\end{eqnarray*}
Thus it suffices to prove the following claim.

\begin{claim}
$\sum_{m=0}^{\infty} [\sym^{n_i+rm} (C,\sigma)] t^{rm}$ is rational
for each $n_i$.
\end{claim}
\begin{proof}
For notational convenience let $n=n_i$. Using
Claim~\ref{claim:symbreak}, we get
\begin{eqnarray*}
 \sum_{m=0}^{\infty} [\sym^{n+rm} (C,\sigma)] t^{rm} & = & \sum_{m=0}^{\infty}[\sym^n (C,\sigma)]t^{rm} \\
 & & +\sum_{m=1}^{\infty} \sum_{j=1}^{r} [E_0,\frac{1}{\lambda_j}\otimes \sigma_0] \sum_{i=1}^m\Omega_{ij}t^{rm}.
\end{eqnarray*}
It is easy to see that the first summand is rational.  Indeed,
\begin{eqnarray*}
 \sum_{m=0}^{\infty}[\sym^n (C,\sigma)]t^{rm} & = & [\sym^n (C,\sigma)]t^r \sum_{m=0}^{\infty}t^m \\
 & = & \frac{[\sym^n (C,\sigma)]t^r}{1-t}.
\end{eqnarray*}
We rewrite the second summand in the following way,
\[ \sum_{m=1}^{\infty}\ \sum_{j=1}^{r} [E_0,\frac{1}{\lambda_j}\otimes \sigma_0] \sum_{i=1}^m\Omega_{ij}t^{rm} =
\sum_{j=1}^{r} [E_0,\frac{1}{\lambda_j}\otimes
\sigma_0]\sum_{m=1}^{\infty}  \sum_{i=1}^m\Omega_{ij}t^{rm}.\] It
suffices to show that the following power series is rational
\[\sum_{m=1}^{\infty} \sum_{i=1}^m\Omega_{ij}t^{rm} \text{ for }j=1,2,\ldots, r.\]
This will be shown with the following series of equalities.
\[ \sum_{m=1}^{\infty} \sum_{i=1}^m\Omega_{ij}t^{rm} = \sum_{m=1}^{\infty}\sum_{i=1}^m [\mathbb{A}^1,1]^{m-i}\prod_{k = j+1}^r[\mathbb{A}^1,\frac{1}{\lambda_j}\otimes \lambda_{k}]^m t^{rm}  \]
\[ =\sum_{m=1}^{\infty} \left(1 + [\mathbb{A}^1,1] + \ldots + [\mathbb{A}^1,1]^{m-1}\right)\left(\prod_{k = j+1}^r[\mathbb{A}^1,\frac{1}{\lambda_j}\otimes \lambda_{k}]\right)^mt^{rm}.\]
Setting $\Omega = \prod_{k = j+1}^r[\mathbb{A}^1,\frac{1}{\lambda_j}\otimes \lambda_{k}]$, we get 
\begin{eqnarray*}
\sum_{m=1}^{\infty} \sum_{i=1}^m\Omega_{ij}t^{rm} & = & \sum_{m=1}^{\infty} \left(1 + [\mathbb{A}^1,1] + \ldots + [\mathbb{A}^1,1]^{m-1}\right)\Omega^mt^{rm} \\
 & = & \sum_{m=1}^{\infty} \Omega^mt^{rm} + [\mathbb{A}^1,1] \sum_{m=2}^{\infty} \Omega^mt^{rm} + [\mathbb{A}^1,1]^2 \sum_{m=3}^{\infty} \Omega^mt^{rm} + \ldots \\
 & = & \Omega t^r\big( \sum_{m=0}^{\infty}\Omega^mt^{rm} + \Omega[\mathbb{A}^1,1]t^r\sum_{m=0}^{\infty}\Omega^mt^{rm} +   \\
 &  & \Omega^2[\mathbb{A}^1,1]^2 t^{2r}\sum_{m=0}^{\infty}\Omega^mt^{rm}+\ldots \big) \\
 & = & \Omega t^r \sum_{m=0}^{\infty} (\Omega t^r)^m \sum_{k=0}^{\infty} (\Omega [\mathbb{A}^1,1]t^r)^k  \\
 & = & \Omega t^r \left(\frac{1}{1 - \Omega t^r} \right)\left(\frac{1}{1- \Omega [\mathbb{A}^1,1] t^r}\right).
\end{eqnarray*}
This concludes the proof of the claim, which completes the proof of
the main theorem.
\end{proof}

\section{Comments about a Generalization}

The assumption that $G$ is abelian is unneeded for most of the proof
of this theorem.  For this reason, one can conjecture that this result generalizes to any finite group.  I believe that the obstacle in proving the general
result can be summed up with the following conjecture:

\begin{conj}
Let $V$ be a vector space over $k$, and let $\sigma$ be a linear
action of $G$ on $V$.  Then the power series \[ \sum_{n=0}^{\infty}
\mathbb{P}( (V,\sigma)^{\oplus r} )t^r \] is rational in the
Grothendieck ring for varieties with group actions.
\end{conj}

The techniques used to prove this conjecture could give us the
necessary tools to apply Corollary~\ref{cor:regular} to the
Picard bundle when the group is non-abelian.  In the proof given
above we used the fact that $G$ is abelian to diagonalize the group
action on the extra part of the vector bundle.  This allowed us to
break up the Picard bundle in a nice way.  If we could find another
way to break up the Picard bundle, without diagonalizing the
representation, we should be able to prove rationality using that method.  This would result in a more general result.

\begin{conj}
Let $G$ be a finite group of order $r$, let $C$ be a non-singular projective curve over an algebraically closed field $k$ of characteristic $0$ or of positive characteristic $p$ with $p \nmid r$, and let $\sigma : G\times C \rightarrow C$ be a group action on $C$.  Then the motivic zeta function
\[ \zeta_{(C,\sigma)}(t) = \sum_{n=0}^{\infty} [\text{Sym}^n (C,\sigma)]t^n\] is rational.
\end{conj}

\newpage

\noindent {\bf Acknowledgements}

\bigskip

I would like to thank my advisor, Valery Lunts, for helping me learn algebraic geometry from the ground up and for his constant guidance while working on my thesis from which this result comes.

\end{document}